\documentclass[12pt]{amsart}

\usepackage {amsmath}
\usepackage[utf8]{inputenc}
\usepackage[english]{babel}
\usepackage {amsfonts, amssymb}
\usepackage {latexsym}

\usepackage{amsthm}

\theoremstyle{plain}
\newtheorem{thm}{Theorem}[section]

\newtheorem{lem}[thm]{Lemma}
\newtheorem{prop}[thm]{Proposition}

\theoremstyle{rem}
\newtheorem {rem} [thm]{Remark}

\newtheorem{defn}[thm]{Definition}

\numberwithin{equation}{section} 
\newcommand {\R} {\mathbb{R}}

\begin{document}

\title[Inextensible strings]{Generalized solutions for inextensible string equations}

\author[Y. \c Seng\"ul]{Yasemin \c Seng\"ul}
\address{Ozyegin University, Department of Natural and Mathematical Sciences, 34794 Alemda\u{g}, Istanbul, Turkey.}
\email{yasemin.sengul@ozyegin.edu.tr}

\author[D. Vorotnikov]{Dmitry Vorotnikov}
\address{CMUC, Department of
Mathematics, University of Coimbra, 3001-501 Coimbra, Portugal.}
\email{mitvorot@mat.uc.pt}

\thanks{We thank Jos\'{e} Miguel Urbano for continuous support and encouragement. We are grateful to Constantine Dafermos, Barbara Lee Keyfitz, David Kinderlehrer,  L\'{e}onard Monsaingeon, Evgeniy Panov, Ivan Yudin and Arghir Zarnescu for stimulating discussions and/or correspondence. 
}

 %the year 2000 is also optional

%%%--------------------------------------------------------------------------------------
\begin{abstract}
We study the system of equations of motion for inextensible strings. This system possesses many internal symmetries, and is related to discontinuous systems of conservation laws and the total variation wave equation.   We prove existence of generalized Young measure solutions with non-negative tension after transforming the problem into a system of conservation laws and approximating it with a regularized system for which we obtain uniform estimates of the energy and the tension. We also discuss sufficient conditions for non-negativity of the tension for strong solutions. 

\end{abstract}

\maketitle
%%%--------------------------------------------------------------------------------------
\section{Introduction}\label{sec:intro}

An inextensible string is defined (cf. \cite{Antman}) to be the one for which the stretch is constrained to be equal to $1$, whatever system of forces is applied to it. As in \cite{Reeken-77}, some authors refer to it as a \emph{chain} which is a long but very thin material that is inextensible but completely flexible, and hence mathematically described as a rectifiable curve of fixed length. Dynamics of pipes, flagella, chains, or ribbons of rhythmic gymnastics, mechanism of whips, and galactic motion are only a few phenomena and applications that can be related to inextensible strings (see \cite{BHT,Hanna-Sant-2013,Dickey} for more details).

The motion executed by a homogeneous, inextensible string with unit length and density can be modeled by the system
\begin{equation}\label{e:whip-pde}
\left\{\begin{array}{ll}
\eta_{tt}(t,s) = \big(\sigma(t,s)\,\eta_{s}(t,s) \big)_{s} + g, \,\,s \in [0,1],\\
| \eta_{s} | = 1,
\end{array}\right.
\end{equation}
where $g \in  \mathbb{R}^{3}$ is the given gravity vector, $\eta \in \mathbb{R}^{3}$ is the unknown position vector for material point $s$ at time $t$. The unknown scalar multiplier $\sigma$, which is called \emph{tension}, satisfies the equation 
\begin{equation} \label{e:eqs}
\sigma_{ss}(t, s) - | \eta_{ss}(t, s) |^{2}\,\sigma(t, s) + | \eta_{st}(t, s) |^{2} = 0
\end{equation}
(see Section \ref{s:eqn-sigma} for the derivation of \eqref{e:eqs} from \eqref{e:whip-pde}).
 We are given the initial positions and velocities of the string as
\begin{equation}
\eta(0,s) = \alpha(s)\,\,\,\text{and}\,\,\,\eta_{t}(0,s) = \beta(s). \label{e:ic}
\end{equation} 
There are several options for boundary conditions:\\ 

a) two fixed ends:
\begin{equation}\label{e:consf}
\eta(t,0) = \alpha(0)\,\,\,\text{and}\,\,\,\eta(t,1) =  \alpha(1)
 \end{equation} 

b) two free ends:
\begin{equation}\label{e:condsf}
\sigma(t,0) =\sigma(t,1) = 0
 \end{equation}
 
c) the ``ring'' or periodic conditions (here it is convenient to consider $s\in \mathbb{R}$ instead of $s\in [0,1]$): 
\begin{equation}\label{e:consr}
\eta(t,s) = \eta(t,s+1)\,\,\,\text{and}\,\,\,
\sigma(t,s) =\sigma(t,s+1)
 \end{equation} 

d) the ``whip'' boundary conditions when one end is free and one is fixed: 
\begin{equation}\label{e:conds}
\sigma(t,0) = 0\,\,\,\text{and}\,\,\,\eta(t,1) = 0. 
\end{equation} 
We make the convention that $s = 0$ corresponds to the free end while the end $s = 1$ is fixed at the origin of the space.

Even though the analysis of the dynamics of inextensible strings subject to different kinds of boundary conditions is a notable problem which goes back to Galileo, Leibniz and Bernoulli  (cf. \cite{Reeken-77,Antman,Preston-2011}), and it has been investigated by many authors in various contexts (see e.g. \cite{BHT,Hanna-Sant-2012,Hanna-Sant-2013,Serre,Wong}), there are still very few results about general well-posedness. One of the existence results available is by Reeken \cite{Reeken-I,Reeken-II} who proves well-posedness for an infinite string with gravity when the initial data is near the trivial (downwards vertical) stable stationary solution (close in $H^{26}$). 

Another one is due to Preston \cite{Preston-2011} who considers \eqref{e:whip-pde} in the absence of gravity with the whip boundary conditions \eqref{e:conds}. He obtains local existence and uniqueness in weighted Sobolev spaces for which the energy is bounded. He uses the method of lines, approximating with a discrete system of chains. In his paper, he imagines that the graph of the whip extends smoothly through the origin (which corresponds to the fixed end), and hence the tension extends to an even smooth function. This evenness leads to what he calls the compatibility boundary condition given by 
\begin{equation*}\label{e:comp-cond}
\sigma_{s}(t, 1) = 0.
\end{equation*}
This condition is related to the non-negativity of the tension, which we discuss in Section \ref{s:eqn-sigma}. We consider all possible boundary conditions together with some conditions on the presence of the gravity. To our knowledge, this is the only proof available for the non-negativity of the tension in this context.

There exist some papers where the problem is investigated from a more geometric point of view. In \cite{Preston-2012}, for example, Preston studies on the space of arcs parametrized by the unit speed in the $L^{2}$ metric. He ignores the gravity and extends the curve through the fixed point by oddness to get a curve with two free endpoints.  He notes that if periodic boundary conditions were used, the results of his paper would change, for example, he would work on ordinary Sobolev spaces on the circle, rather than weighted Sobolev spaces on the interval. Dickey \cite{Dickey}, on the other hand, looks into two dimensional case, also ignoring the gravity. He defines a new variable as the angle the tangent to the string makes with the positive $x$-axis, and obtains a transformed system for which he discusses two asymptotic theories, one in which the amplitude of the angle is small and another in which the amplitude is large. 

%As mentioned in \cite{Dickey} problem \eqref{e:whip-pde} has a variational formulation which can be stated as minimizing the functional
%\begin{equation*}\label{e:min-func}
%I(\eta) = \frac{1}{2} \int_{0}^{T} \int_{0}^{1} \left( | \eta_{t} |^{2} - \sigma(s,t) | \eta_{s} |^{2} \right) ds dt 
%\end{equation*} 
%where $\sigma(s, t)$ is chosen so that $|\eta_{s} | = 1$ is satisfied.

After certain transformations of \eqref{e:whip-pde} (see Section \ref{ss32}) we obtain the hyperbolic system of conservation laws in \eqref{e:transwithg}. This kind of systems are mentioned in the book by Dafermos \cite[Chapter 7]{Dafermos-book} as examples of balance laws in one space dimension arising in the contexts of planar oscillations of thermoelastic medium and oscillations of flexible, extensible elastic strings. To our knowledge, there is no existence result in this context for conservation laws as well as for the $1$-Laplacian wave equation \eqref{e:one-laplace} which is derived from \eqref{e:whip-pde} by certain transformations (see Section \ref{ss32}).   The difficulty of the problem is not surprising since the system  of conservation laws \eqref{e:claw} is not strictly hyperbolic, and its flux is discontinuous at zero. 

Scalar hyperbolic conservation laws
with a discontinuous flux were recently considered in \cite{Bulicek-etal}. Although the authors of that paper notice that their procedures do not work in the case of systems, we managed to find a slightly similar approach in the case of our particular system \eqref{e:claw}. Note that a related but different class of problems concerns scalar conservation laws with a flux that is discontinuous in the spatial and not in the unknown variable \cite{au2005,panov}.  

We also point out that although our problem has formally nothing to do with fluids, its mathematical and physical behaviours resemble
those of an ideal fluid described by the 3D
Euler equations \cite{tzn,Preston-2011}. Thus, the studies of this ``toy model'' may shed more light on the nature of turbulence \cite{co06}.

The novelty of our paper is three-fold. Firstly, we show global existence of solutions for the equations of motion of the inextensible string without restrictions on the initial data. Secondly and thirdly, this seems to be the first treatment of well-posedness both for the systems of hyperbolic conservation laws with discontinuous flux and for the total variation wave equation. 

We will work with the most complex boundary conditions, namely the ``whip" conditions \eqref{e:conds}, but the results of the paper remain valid for any of \eqref{e:consf},  \eqref{e:condsf} or \eqref{e:consr}. In some places throughout the paper we emphasize the technical differences of those cases with respect to \eqref{e:conds}. Moreover, the three-dimensional space was chosen due to the physical meaning of the problem, but, mathematically, everything presented in the paper is true in any dimension.  

The paper is organized in the following way. In Section 2.1, we introduce the basic notation. In Section 2.2, we make a series of transformations of our problem and obtain a system of hyperbolic conservation laws with discontinuous flux and the total variation wave equation. In Section 2.3, we derive an equivalent system which is more tractable due to the lack of discontinuity. In Section 2.4, we discuss the non-negativity of the tension which is crucial in our considerations. In Section 2.5, we make some preliminary observations related to the energy. In Section 3.1, we recall the main concepts of the theory of generalized Young measures. In Section 3.2, we define the generalized solutions to our system of conservation laws with discontinuous flux. In Section 4.1, we introduce an approximate problem and study its global well-posedness.  In Section 4.2, we define the energy for the approximate problem, and show dissipativity of that problem. This allows us to derive, in Section 4.3, a crucial uniform $L^1$-bound for the tension. In Section 5.1, we prove the main result of the paper, which is the existence of generalized Young measure solutions with non-negative tension to the initial-boundary value problem for the equations of motion of the inextensible string, employing the equivalent continuous formulation introduced in Sections 2.3 and 3.2. Then, in Section 5.2, we touch on some examples which illustrate our key finding that even though strong solutions with non-negative tension do not exist for some initial data, generalized solutions with non-negative tension do exist and can be interpreted from the point of view of mechanics.

\section{Preliminaries}

\subsection{Some conventions}

Throughout the paper we will denote $\Omega = (0,T) \times (0, 1).$ The scalar product of any two vectors $\chi, \xi$ in $\mathbb{R}^3$ is simply denoted by $\chi\xi$, and $|\chi|$ is the Euclidean norm $\sqrt{\chi\chi}$. The notation $Lip_1([0,1];\mathbb{R}^3)$ stands for the set of continuous functions $f:[0,1]\to\mathbb{R}^3$ satisfying 
\[|f(s_1)-f(s_2)|\leq |s_1-s_2|, \ s_1,s_2\in[0,1].\]
The symbol $S^{n-1}$ stands for the unit sphere in $\mathbb{R}^n$, $n\in \mathbb{N}$. $M^+(U)$ and $M^1(U)$ are the spaces of positive finite and probability
measures, respectively, on a closed set $U\subset \mathbb{R}^n.$ $L_{w}^{\infty}(U_1, \mu; \mathcal{M}^{1}(U))$ is the space of $\mu$-weakly*-measurable maps (cf. \cite{Sze}) from an open or closed set
$U_1\subset \mathbb{R}^m$ into $M^1(U)$ (the default measure $\mu$ on $U_1$ is the Lebesgue measure). Generic positive constants are denoted by $C$.  Finally, by \emph{regular solutions} in various contexts we mean sufficiently smooth functions so that all derivatives involved in the associated arguments are continuous.

\subsection{Changes of variables and formal transformations} \label{ss32}\hfill
\vspace{2mm}

\noindent 1. We make an ansatz that $\sigma\geq 0$ (cf. the discussion in Section \ref{s:eqn-sigma}). By putting $\kappa := \sigma\,\eta_{s}$ we get $\sigma =  | \kappa |$ and $\eta_{s} = \frac{\kappa}{| \kappa |}.$
We can then formally rewrite \eqref{e:whip-pde} as
\begin{equation}\label{e:transwithg}
\begin{array}{ll}
{\left\{\begin{array}{ll}
\eta_{tt} = \kappa_{s}+g, \\ \eta_{s} = \displaystyle\frac{\kappa}{| \kappa |}
\end{array}\right.} \qquad \underset{(v := \eta_{t})}{\Longleftrightarrow} \qquad 
{\left\{\begin{array}{ll} v_{t} = \kappa_{s}+g
, \\ v_{s} = \displaystyle\left(\frac{\kappa}{| \kappa |}\right)_{t}
\end{array}\right.}
\end{array}
\end{equation} 
From  \eqref{e:conds} we infer that the boundary conditions for $\kappa$ take the form
\begin{equation}\label{e:bckappa}
\kappa(t,0) = 0\qquad \text{and} \qquad \kappa_{s}(t,1) = - g.
\end{equation}
The second condition follows from 
\begin{equation}\label{e:icv1}
v(t,1) = 0
\end{equation}
(the velocity of the fixed end is zero). 
Note that we can find the initial conditions for $\kappa$ (and $\sigma$) using  \eqref{e:eqs}, \eqref{e:ic} and first condition in \eqref{e:bckappa}. We also observe that 
\begin{equation}\frac{\kappa}{| \kappa |} (0,s)= \eta_s(0,s)=\alpha_{s}(s) \end{equation} 
and \begin{equation}\label{e:vic}v(0,s)= \eta_t(0,s) = \beta(s). \end{equation}

\noindent  2. If $\alpha(1)=0$, then by using \begin{equation}\label{e:eta1} \eta(t,s) = \alpha(s) + \int_{0}^{t} v(r,s)\,dr \quad \text{and} \quad   \sigma=|\kappa|,\end{equation} we can come back from the ``velocity $v$ 
-- contact force $\kappa$'' formulation \eqref{e:transwithg}--\eqref{e:vic} to the original ``position $\eta$ -- tension $\sigma$'' setting \eqref{e:whip-pde}, \eqref{e:ic}, \eqref{e:conds}.

\noindent  3. Let $\Upsilon=(v,\kappa)\in\mathbb{R}^6$, and define the map $F:\mathbb{R}^6\times[0,T]\to\mathbb{R}^6$, $(v,\kappa,t)\mapsto \left(\frac{\kappa}{| \kappa |},v - g t\right)$. Then \eqref{e:transwithg} can be rewritten in the form 
\begin{equation} 
 \label{e:claw} \Upsilon_s=[F(\Upsilon,t)]_t. 
\end{equation}
This is a system of conservation laws with discontinuous flux $F$, where $s$ plays the role of time and $t$ plays the role of space. 

\noindent  4. Let us now further define 
\begin{equation}\label{e:phi}
\phi(t,s) : = \int_{0}^{t} \kappa(z,s)\,dz.
\end{equation}
From \eqref{e:transwithg} we get
\[\int_{0}^{t} v_{t} \,dt = \int_{0}^{t} \kappa_{s}\,dt + \int_{0}^{t} g \,dt\]
which, by \eqref{e:vic} and \eqref{e:phi}, gives
\begin{equation}\label{e:forsigmabound}
v = \phi_{s} + g\,t + \beta.
\end{equation}
Together with \eqref{e:transwithg} this leads to 
\begin{equation}\label{e:one-laplace}
\phi_{ss} (t,s) + \beta_{s} = \left(\frac{\phi_{t}}{| \phi_{t} |}\right)_{t} = \Delta_{1} \phi(t,s).
\end{equation}
This is the $1$-Laplacian wave equation which can also be called the \emph{total variation wave equation}. The latter in a different context might be relevant in image processing, cf. \cite{va04,chan05}. Here, once again, $s$ plays the role of time and $t$ plays the role of space. 

The initial/boundary conditions for $\phi$ are
\begin{subequations}\label{e:conds-phi}
\begin{align}
& \phi(t,0) = 0\,\,\,\text{and}\,\,\,\phi_{s}(t,1) = - g\,t, \label{e:bc-phi}  \\
& \phi(0,s) = 0\,\,\,\text{and}\,\,\,\phi_{t}(0,s) = \kappa(0,s) \label{e:ic-phi}.
\end{align}
\end{subequations}
5. Since $|\eta_s|=1$, a necessary assumption for existence of regular solutions is \begin{equation}\label{e-al1}|\alpha_s|=1.\end{equation} Differentiating the equation $|\eta_s|^2=1$ with respect to time we get $\eta_s\eta_{st}=0$, yielding the second necessary condition \begin{equation}\label{e-al2} \alpha_s\beta_s=0.\end{equation}
  
\subsection{Removing the discontinuity} \label{ss25} By the transformation $u = \eta_{s} \sqrt{\sigma}, v = \eta_{t}$ we obtain
\begin{equation} 
{\left\{\begin{array}{ll}\label{e:quadr-eqn}
v_{t} = \big(u\,| u | \big)_{s}+g, \\ v_{s} = \displaystyle\left(\frac{u}{| u |}\right)_{t}
\end{array}\right.}
\end{equation}
Defining $\xi = (v,u)$ we can rewrite this in the form
\begin{equation*}
\Phi(\xi)_{t} = \Psi(\xi)_{s}+(g,0) \quad \text{where} \quad 
{\left\{\begin{array}{ll}
\Phi(\xi) = \displaystyle\left(v,\frac{u}{| u |} \right) \\ \Psi(\xi) = \big(u\,| u |,v \big)
\end{array}\right.}
\end{equation*}
Let $\mathcal{P}:\mathbb{R}^6\to\mathbb{R}^6$ be the projection $\xi\mapsto (0,u)$.

Inspired by the implicit constitutive theory (cf. \cite{Bulicek-etal}), we proceed as follows. We define 
\[\Gamma(\xi) := \left(v,\frac{u}{| u |} + u\right). \]
Then, for $\gamma=(v,w)$, we consider 
\[\Gamma^{-1}(\gamma) = \left(v,\mathcal{M}(w) \right),\]
where 
\[\mathcal{M}(w) : = {\left\{\begin{array}{ll}
0 \quad \text{for}\quad | w | \leq 1, \\ w - \frac{w}{| w |}\quad \text{for}\quad | w | \geq 1.
\end{array}\right.}\]
Taking the derivative of $\Phi(\xi)$ with respect to time, we find
$$\Gamma(\xi)_{t} - \mathcal{P}\xi_{t} = \Psi(\xi)_{s}+(g,0),$$ whence $$ \Gamma(\Gamma^{-1}(\gamma))_{t} -  [\mathcal{P}(\Gamma^{-1}(\gamma))]_{t} = \Psi(\Gamma^{-1}(\gamma))_{s}+(g,0).$$
We formally conclude that
\[\gamma_{t} - [\mathcal{P}(\Gamma^{-1}(\gamma))]_{t} = \Psi(\Gamma^{-1}(\gamma))_{s}+(g,0).\]
Defining 
\[{\left\{\begin{array}{ll}
\mathcal{A}(\gamma) = \gamma - \mathcal{P}(\Gamma^{-1}(\gamma)),\\ \mathcal{B}(\gamma) = \Psi(\Gamma^{-1}(\gamma))
\end{array}\right.}\]
we obtain
\begin{equation}\label{e:cont-eqn}
\mathcal{A}(\gamma)_{t} =\mathcal{B}(\gamma)_{s}+(g,0).
\end{equation}

Rather surprisingly, the new equation is equivalent to the original system \eqref{e:whip-pde} coupled with the additional restriction \begin{equation} \label{e:sg0} \sigma\geq 0,\end{equation} provided the solutions are regular and some natural compatibility conditions hold.  Indeed, it is straightforward to check that for any solution $(\eta,\sigma)$ of the system \eqref{e:whip-pde},\eqref{e:sg0}, the corresponding vector function $$\gamma=(v,w)=(\eta_t,\eta_s(1+\sqrt\sigma))$$ satisfies \eqref{e:cont-eqn}. 

We now assume that \begin{equation}\label{e:ca1} |w(0,s)|\geq 1,\ w(0,s) v_s(0,s)=0 \end{equation} for all $s\in[0,1]$. In Section \ref{ss72} we will realize that this is a necessary and legitimate assumption. We now take any regular solution $\gamma=(v,w)$ to \eqref{e:cont-eqn}. At the relative interior of the set $\{(t,s)\in\overline\Omega:\ |w(t,s)|\geq 1\}$,  
letting $\kappa=
\frac w {|w|}
(|w|-1)^2,
$ we obtain \begin{equation} \label{e:vkap}
{\left\{\begin{array}{ll} v_{t} = \kappa_{s}+g
, \\ v_{s} = \displaystyle\left(\frac{w}{| w |}\right)_{t}
\end{array}\right.}
\end{equation} 
Since
\[1= \left| \frac{w}{| w |} \right|^{2},\]
differentiating it with respect to time gives
\begin{equation} \label{e:ort} 0 = \left( \frac{w}{| w |} \right)_{t}\,\frac{w}{| w |} = v_{s}\,\frac{w}{| w |}.\end{equation}
Assume that there is a point $(t_0,s_0)$ such that $|w(t_0,s_0)|<1$. Let $\Omega_1$ be the connected component of the set $\{(t,s)\in\overline\Omega:\, |w(t,s)|<1\}$ containing $(t_0,s_0)$, and let $t_1=\inf\{t\geq 0: (t,s_0)\in\Omega_1\}$. If $t_1=0$ then \begin{equation}\label{e:ca2} |w(t_1,s_0)|\geq 1,\ w(t_1,s_0) v_s(t_1,s_0)=0 \end{equation}  due to \eqref{e:ca1}, and if $t_1>0$ then \eqref{e:ca2} follows from $\eqref{e:ort}$ by continuity. For $(t,s_0)\in \Omega_1$, the solution to \eqref{e:cont-eqn} can be written explicitly as
\[v(t,s_0)=(t-t_1)g+v(t_1,s_0),\, w(t,s_0)=(t-t_1)v_s(t_1,s_0)+w(t_1,s_0).\] By the Pythagorean theorem, \[|w(t,s_0)|\geq |w(t_1,s_0)|\geq 1,\] arriving at a contradiction.

Consequently, $|w|\geq 1$ on $\overline\Omega$, and thus \eqref{e:vkap} holds everywhere. By \eqref{e:vkap}, there exists a vector function $\eta$ such that $\eta_s=\frac w {|w|}$ and $\eta_t=v$. This function $\eta$ solves the system \eqref{e:whip-pde},\eqref{e:sg0} with $\sigma=(|w|-1)^2$. Note that $\eta$ is determined up to a constant unless initial or boundary conditions are specified. 

By the above analysis, we have killed the discontinuity since $\mathcal{A}$ and $\mathcal{B}$ are both continuous with $\mathcal{A}$ being sublinear and $\mathcal{B}$ having at most quadratic growth. We will therefore proceed in the same way amid the weak formulation of our problem in Section \ref{ss72}.

\subsection{The equation for the tension}\label{s:eqn-sigma}

Differentiating the constraint $| \eta_{s} |^{2} = 1$ with respect to $s$ shows that $\eta_s\eta_{ss}=0$. 
Hence, multiplying the first equation in \eqref{e:whip-pde} by $\eta_s$ we get
\begin{equation}\label{e:estr}\eta_{s} \,\eta_{tt} = \sigma_{s}+g\eta_s.\end{equation}
Now, differentiating $| \eta_{s} |^{2} = 1$ twice with respect to time we obtain
\[\eta_{s} \, \eta_{stt} + \eta_{st} \, \eta_{st} = 0.\]
Due to \eqref{e:estr},
\[\eta_{ss}\,\eta_{tt} + \eta_{s}\,\eta_{stt} = \sigma_{ss}+g\eta_{ss}.\]
Combining these two equations we get
\[\sigma_{ss} - (\eta_{tt}-g) \eta_{ss}+ | \eta_{st} |^{2} = 0.\]
Expressing $(\eta_{tt}-g)$ by \eqref{e:whip-pde}, we end up with
\begin{equation}\label{e:sigma-eqn}
\sigma_{ss} - | \eta_{ss} |^{2}\,\sigma + | \eta_{st} |^{2} = 0.
\end{equation}

\begin{prop}
Let $(\eta,\sigma)$ be a regular solution to \eqref{e:whip-pde}, \eqref{e:ic} with one of the boundary conditions \eqref{e:consf}--\eqref{e:conds}.  Assume that one of the following assumptions holds:\\
\indent \emph{(i)}\,\, the boundary condition is \eqref{e:condsf} or \eqref{e:consr};\\
\indent \emph{(ii)}\,\, the boundary condition is \eqref{e:conds} and $g= 0$;\\
\indent \emph{(iii)}\,\,the boundary condition is \eqref{e:consf}, $|\alpha(0)-\alpha(1)|<1$ and $g= 0$.\\
Then $\sigma \geq 0$ for all times.
\end{prop}
\begin{proof}
Assume that, for some $t$, the minimum of $\sigma(t,\cdot)$ is negative.  Note that from \eqref{e:sigma-eqn} we have
\[\sigma\,| \eta_{ss} |^{2}  -  \sigma_{ss} \geq 0.\]
By the maximum principle \cite{Protter1984}, either $\sigma(t,\cdot)$ is a negative constant, or the minimum is achieved at $s=0$ or $1$. 

The first alternative is impossible for \eqref{e:condsf} and \eqref{e:conds}, and in the remaining cases it implies $\,| \eta_{ss}(t,\cdot) |\equiv 0$, so the string should be straight, and thus   \[|\eta(t,0)-\eta(t,1)|=1.\]  This obviously contradicts \eqref{e:consr}, whereas \eqref{e:consf} would yield $|\alpha(0)-\alpha(1)|=1$. 

The second alternative can only hold \cite{Protter1984} provided $\sigma_s(t,0)> 0$ (if the minimum is at $0$) or $\sigma_s(t,1)< 0$ (if the minimum is at $1$). This immediately rules out the periodic case, so the negative minimum can only be achieved at fixed ends. But \eqref{e:estr} implies that at such points $\sigma_s=-g\eta_s$, and we again arrive at a contradiction.
\end{proof}

This proof implies that, for the ``whip'' boundary condition \eqref{e:conds}, instead of assuming that the gravity is zero, it suffices to know a priori that $g\eta_s(t,1) \leq 0$, whereas, for two fixed ends \eqref{e:consf}, it suffices to know a priori that $g\eta_s(t,0) \geq 0$ and $g\eta_s(t,1) \leq 0$. We believe that there exist much weaker hypotheses which guarantee non-negativity of the tension. Our conjecture is that, for both \eqref{e:consf} and \eqref{e:conds}, if $\sigma_0(s):=\sigma(0,s)\geq 0$ for all $s\in [0,1]$, then $\sigma\geq 0$ on $\overline\Omega$. Remember that $\sigma_0$ is determined by $\alpha$, $\beta$ and the boundary conditions. For example, in the ``whip'' case \eqref{e:conds} it is the solution of the problem \begin{equation}({\sigma_0})_{ss} - | \alpha_{ss} |^{2}\,\sigma_0 + | \beta_{s} |^{2} = 0, \ \sigma_0(0)=0,\ (\sigma_0)_s(1)=-g\alpha_s(1).\end{equation} 
However, for non-zero gravity, $\sigma$ can be negative at the initial moment of time and even for all times. For instance, \eqref{e:whip-pde} has an unstable stationary solution \begin{equation}\label{e:ss} \eta_u(s)=\alpha_u(s)=(s-1)\frac g{|g|},\ \sigma_u(s)=-|g|s,\end{equation}
which satisfies both  \eqref{e:consf} and \eqref{e:conds}.

Nevertheless, our ansatz $\sigma\geq 0$ is meaningful even for such ``unstable'' problems as \eqref{e:whip-pde}, \eqref{e:ic}, \eqref{e:conds} with the initial data \begin{equation} \label{e:unst} \alpha=\alpha_u,\,\beta=0. \end{equation} There exist objects which can be interpreted as generalized solutions to this problem with non-negative tension. We will get back to this example in Section \ref{s:52}.

\subsection{Conservation of energy}

We define the kinetic and potential energies as 
\begin{equation}\label{e:energies}
K(t) = \frac{1}{2} \int_{0}^{1} |\eta_{t}|^{2}\,ds \quad \text{and} \quad P(t) = - \int_{0}^{1} g\,\eta\,ds.
\end{equation}

\begin{prop}\label{l:energy-cons} Let $(\eta,\sigma)$ be a regular solution to \eqref{e:whip-pde}, \eqref{e:ic} with one of the boundary conditions \eqref{e:consf}--\eqref{e:conds}. Then
the total energy $E(t) := K(t) + P(t)$ is conserved.
\end{prop}
\begin{proof}
From \eqref{e:whip-pde} we have 
\begin{eqnarray*}
& & \frac{d}{dt} (K(t)+P(t))=\frac{1}{2} \frac{d}{dt} \int_{0}^{1} |\eta_{t}|^{2}\,ds  -\int_{0}^{1} g\,\eta_{t}\,ds= \\
&& \qquad =  \int_{0}^{1} \eta_{tt}\,\eta_{t}\,ds -\int_{0}^{1} g\,\eta_{t}\,ds  =\int_{0}^{1} (\sigma\eta_s)_{s}\,\eta_{t}\,ds\\
&& \qquad = \sigma(t,1)\,\eta_{s}(t,1)\eta_{t}(t,1) - \sigma(t,0)\,\eta_{s}(t,0)\eta_{t}(t,0)- \int_{0}^{1}\sigma\, \eta_{s}\eta_{ts}\,ds.
\end{eqnarray*}
The third term is identically zero as observed in the end of Section \ref{ss32}. The first two terms vanish if $0$ and $1$ are either free or fixed ends. In the periodic case their difference is still zero. 
\end{proof}

In the absence of the gravity, as also mentioned in \cite{Dickey}, the energy of the whip is entirely kinetic due to the fact that it cannot stretch or compress and hence cannot store energy. In this case from Proposition \ref{l:energy-cons} we obtain that 
\begin{equation*}
\int_{0}^{1} | \eta_{t}(t,s) |^{2} \,ds = \int_{0}^{1} | \eta_{t}(0, s) |^{2} \,ds=2E(0),\ t > 0 .
\end{equation*}
In the general case, we have 
\begin{eqnarray} \label{e:qaestim}
\frac 1 2 \int_{0}^{1} | \eta_{t}(t,s) |^{2} \,ds &=& \frac 1 2 \int_{0}^{1} | \eta_{t}(0, s) |^{2} \,ds + \int_{0}^{1} g(\eta-\alpha)\,ds \nonumber \\
 &=& E(0)+ \int_{0}^{1} \,g\eta\,ds.
\end{eqnarray}
If the initial energy is finite, with the help of Gr\"onwall's lemma \eqref{e:qaestim} implies
\begin{equation} \label{e:aestim}
\int_{0}^{1} | \eta_{t}(t,s) |^{2} \,ds \leq C,\ t \in [0,T].
\end{equation}
(cf. the reasoning in Section \ref{uee}). When at least one end is fixed, the potential energy is a priori bounded because of $|\eta_{s} | = 1$, and thus $C$ in \eqref{e:aestim} does not depend on $T$. 

\section{Setting in the context of Young measures} \label{sect3}
\subsection{Introduction}

We will essentially follow \cite{Wiedemann-thesis} for a basic introduction to the generalized Young measures. 

Let $m,l,d\in\mathbb{N}$, $p \in [1, +\infty)$, $\Gamma \subset \mathbb{R}^{m}$ be an open set. We define $\mathcal{F}_{p}$ as the collection of continuous functions $f \colon \Gamma \times \mathbb{R}^{l} \to \mathbb{R}^d$ for which the limit 
\[f^{\infty}(x,z) := \underset{\substack{x' \to x \\ z' \to z \\ s \to \infty}}{\lim} \frac{f(x', s z')}{s^{p}}\]
exists for all $(x,z) \in \overline{\Gamma} \times \mathbb{R}^{l}$ and is continuous in $(x,z)$. The function $f^{\infty}$ is called the \emph{$L^{p}$-recession function} of $f$. Note that it is $p$-homogeneous in $z$, i.e., $f^{\infty}(x,r z ) = r^{p} f^{\infty}(x,z)$ for all $r \geq 0$. 

%Examples of functions in $\mathcal{F}_{p}$ are given by continuous functions satisfying $| f(x,z) | \leq C ( 1 + |z|^{q})$ with $0 \leq q < p$, in which case $f^{\infty} = 0,$ or by continuous functions which are $p$-homogeneous in $z,$ in which case $f^{\infty} = f.$ Of course, functions in $\mathcal{F}_{p}$ always satisfy a bound $|f(x,z)| \leq C (1 + |z|^{p})$ (where $C$ might depend on $x$).

A generalized Young measure on $\mathbb{R}^{l}$ with parameters in $\Gamma$ is defined as a triple $(\nu, \lambda, \nu^{\infty})$ such that 
\begin{gather}
\nu \in L_{w}^{\infty}(\Gamma; \mathcal{M}^{1}(\mathbb{R}^{l}) ), \nonumber \\
\lambda \in \mathcal{M}^{+}(\overline{\Gamma}), \nonumber \\
\nu^{\infty} \in L_{w}^{\infty}(\overline{\Gamma}, \lambda; \mathcal{M}^{1}(S^{l-1})). \nonumber
\end{gather}
Note that $\nu$ is defined Lebesgue-a.e. on $\Gamma$, and $\nu^{\infty}$ is defined $\lambda$-a.e. on $\overline{\Gamma}$; $\nu$ is called the \emph{oscillation measure}, $\lambda$ is the  \emph{concentration measure} and $\nu^{\infty}$ is the \emph{concentration-angle measure}.

Now, we can state the fundamental theorem on generalized Young measures (see \cite{Ali-Bouc,DiPerna-Majda,Kris-Rind,Wiedemann-thesis}):

\begin{thm} \label{t:Young}
Let $\{w_{n}\}\subset L^{p}(\Gamma;\mathbb{R}^l)$ be an $L^p$-bounded sequence of maps. Then there exists a subsequence (not relabeled) and a generalized Young measure $(\nu, \lambda, \nu^{\infty})$ such that, for every $f \in \mathcal{F}_{p}$,
\[\int_{\Gamma}f(x, w_{n}(x)) dx \rightarrow \int_{\Gamma}\langle \nu_{x}, f(x, \xi) \rangle dx + \int_{\overline{\Gamma}} \langle \nu_{x}^{\infty}, f^{\infty}(x, \theta) \rangle \lambda(dx),\]
where \[\langle \nu_{x}, f(x, \xi) \rangle = \int_{\mathbb{R}^{l}} f(x, \xi) \nu_{x}(d\xi), \ \langle \nu_{x}^{\infty}, f^{\infty}(x, \theta) \rangle = \int_{S^{l - 1}} f^{\infty}(x, \theta) \nu^{\infty}(d\theta).\]
\end{thm}
%In the situation of the above theorem, we say that the subsequence $(w_{n})$ generates the Young measure $(\nu, \lambda, \nu^{\infty})$ in $L^{p}(\Gamma)$, and occasionally we write $w_{n} \overset{\textbf{\textrm{Y}}}{\rightarrow} (\nu, \lambda, \nu^{\infty})$. We also use the shorthand notation 
%\[\langle\langle \nu, \lambda, \nu^{\infty}; f \rangle\rangle = \int_{\Gamma} \langle \nu, f \rangle dx + \int_{\overline{\Gamma}} \langle \nu^{\infty}, f^{\infty} \rangle d \lambda, \]
%which emphasizes the duality between the space of generalized Young measures and $\mathcal{F}_{p}$.

\begin{rem} \label{remme} {\rm In particular, for $f(x,\xi)=|\xi|^p$ we infer that $$\|w_n\|_{L^p(\Gamma)^l}\to\int_{\Gamma} \langle \nu_{x}, | \xi |^{p} \rangle dx +\lambda(\overline\Gamma)$$ in view of $f^\infty\equiv 1$ on $S^{l-1}$.}\end{rem}

\subsection{Weak setting of the incompressible string problem}

\label{ss72}

Consider the problem of finding a velocity field $v$ and a contact force $\kappa$, which was derived in Section \ref{ss32} from the original problem \eqref{e:whip-pde}, \eqref{e:ic}, \eqref{e:conds}:
\begin{subequations}\label{e:weak-setting}
\begin{align}
& v_{t} = \kappa_{s} + g \label{e:v_t-1}, \\
& v_{s} = \left(\frac{\kappa}{| \kappa |}\right)_{t} \label{e:v_s-1}, \\
& \kappa \vert_{s = 0} = 0 \label{e:bc(s)-kappa-1},\\
& \frac{\kappa}{| \kappa |} \Big\vert_{t = 0} = \alpha_{s} \label{e:bc(t)-kappa-1}, \\
& v \vert_{s = 1} = 0, \label{e:bc(s)-v-1}\\
& v \vert_{t = 0} = \beta.\label{e:bc(t)-v-1}
\end{align}
\end{subequations}
Let us define the auxiliary function $h_{0} \colon \overline{\mathbb{R}}_{+} \to \overline{\mathbb{R}}_{+}$ with $h_{0}(r) = 1 + \sqrt{r}.$ Then we have
\[h_{0}^{-1}(r) = (r - 1)^{2}, \ r \geq 1 ,\]
and we can continue $h_{0}^{-1}$ by zero for $r \leq 1.$ We also define $H_{0}, H_{0}^{*} \colon \mathbb{R}^{3} \to \mathbb{R}^{3}$ as
\begin{equation*}
H_{0}(\chi) = \frac{\chi}{| \chi |} h_{0}^{-1}( | \chi | ), \,\,\, H_{0}(0) = 0, \,\,\, H_{0}^{*}(\chi) = \frac{\chi}{| \chi |} \sqrt{h_{0}^{-1}(| \chi |)}, \,\,\, H_{0}^{*}(0) = 0. 
\end{equation*}
Let \[w = \frac{\kappa}{|\kappa|} + \frac{\kappa}{\sqrt{|\kappa|}} = h_{0}(|\kappa|) \frac{\kappa}{|\kappa|}.\]
Then $\kappa = H_{0}(w)$ and 
 \[\frac{\kappa}{|\kappa|}= \frac{\kappa}{|\kappa|}+\frac{\kappa}{\sqrt{|\kappa|}}-\frac{w}{| w |} \sqrt{|\kappa|}=w-\frac{w}{| w |}  \sqrt{h_{0}^{-1}(|w|)}=w - H_{0}^{*}(w),\] so we can rewrite \eqref{e:v_t-1} and \eqref{e:v_s-1} as 
\begin{subequations}
\begin{align}
& v_{t} = (H_{0}(w))_{s} + g \label{e:v_t-2}, \\
& v_{s} =  \left(w - H_{0}^{*}(w) \right)_{t}. \label{e:v_s-2}
\end{align}
\end{subequations}

In Section \ref{ss25} we showed that this system was equivalent to \eqref{e:whip-pde}, \eqref{e:sg0}. Observe that, in the current setting, \eqref{e:ca1} is a consequence of the compatibility conditions \eqref{e-al1} and \eqref{e-al2}. Indeed, $$|w(0,s)|\geq |w(0,s) - H_{0}^{*}(w(0,s))|=|\alpha_s(s)|=1,$$
$$w(0,s)v_s(0,s)=|w(0,s)|\alpha_s(s)\beta_s(s)=0.$$

Define the space $\tilde{C}^{\infty}(\overline{\Omega})$ of test functions to be the set of pairs $\varphi = (\phi, \psi)$, $\phi, \psi \in C^{\infty}(\overline{\Omega}; \mathbb{R}^{3})$ such that
\begin{subequations}
\begin{align}
& \phi \vert_{s = 1} = 0,\quad \phi_{s}\vert_{s = 0} = 0, \quad \phi \vert_{t = T} = 0, \\
& \psi \vert_{s = 0} = 0,\quad \psi_{s}\vert_{s = 1} = 0, \quad \psi \vert_{t = T} = 0 .
\end{align}
\end{subequations}
Take any $\varphi = (\phi, \psi) \in \tilde{C}^{\infty}(\overline{\Omega})$. Multiplying \eqref{e:v_t-1} (or \eqref{e:v_t-2}) by $\phi$ and integrating in space and time gives
\begin{equation}\label{e:tomerge-1}
\int_{\Omega} v \phi_{t} \,ds \,dt = \int_{\Omega} H_{0}(w) \phi_{s} \,ds \,dt - \int_{0}^{1} \beta \phi\vert_{t = 0} \,ds - \int_{\Omega} g \phi \,ds \,dt.
\end{equation}
Doing the same with \eqref{e:v_s-1} (or \eqref{e:v_s-2}) and $\psi$ gives
\begin{equation}\label{e:tomerge-2}
\int_{\Omega} [w - H_{0}^{*}(w)] \psi_{t} \, ds\,dt = \int_{\Omega} v \psi_{s} \,ds \,dt + \int_{0}^{1} \alpha \psi_s\vert_{t = 0} \,ds.
\end{equation}
Observe that we have taken into account \eqref{e:bc(s)-kappa-1} - \eqref{e:bc(t)-v-1}, and the setting \eqref{e:tomerge-1} - \eqref{e:tomerge-2} already incorporates the initial and boundary conditions. We also used the assumption \begin{equation} \label{e-al3} \alpha(1) = 0.\end{equation}

Denote $\gamma = (v,w) \in \mathbb{R}^{6}$, and define functions $\mathcal{A}, \mathcal{B} \colon \mathbb{R}^{6} \to \mathbb{R}^{6}$ such that
\begin{align}
& \mathcal{A}(\gamma) = \mathcal{A}(v,w) = (v, w - H_{0}^{*}(w)), \\
& \mathcal{B}(\gamma) = \mathcal{B}(v,w) = (H_{0}(w), v),
\end{align}
and also the operator
\begin{eqnarray*}
 \Xi_{0}(\alpha, \beta, \varphi) & = & \Xi_{0}(\alpha, \beta, \phi, \psi) \\
& = &  - \int_{0}^{1} \beta \phi \vert_{t = 0} \,ds + \int_{0}^{1}\alpha \psi_{s} \vert_{t = 0} \,ds - \int_{\Omega} g \phi \,ds\,dt .
\end{eqnarray*}
Then \eqref{e:tomerge-1} and \eqref{e:tomerge-2} can be merged to get
\begin{equation}\label{e:merged}
\int_{\Omega} \mathcal{A}(\gamma) \varphi_{t} \,ds\,dt = \int_{\Omega} \mathcal{B}(\gamma) \varphi_{s}\,ds\,dt + \Xi_{0}(\alpha, \beta, \varphi).
\end{equation}
Observe that $\mathcal{A}$ and $\mathcal{B}$ are in the class $\mathcal{F}_{2}.$ Moreover, since $\mathcal{A}$ is sublinear, $\mathcal{A}^{\infty} \equiv 0,$ whereas it can be checked that $\mathcal{B}^{\infty}(v, w) = (w |w|, 0)$.

These considerations and analogy with \cite{Bre-DL-Sz,DiPerna-Majda,Sze,Wiedemann-thesis} suggest:
\begin{defn}
A triple $(\nu, \lambda, \nu^{\infty})$ with
\begin{align}
&\nu \in L^{\infty}_{w}(\Omega; \mathcal{M}^{1}(\mathbb{R}^{6})), \\
&\lambda  \in \mathcal{M}^{+}(\overline{\Omega}), \\
& \nu^{\infty} \in L^{\infty}_{w}(\overline{\Omega}, \lambda; \mathcal{M}^{1}(S^{5})),
\end{align}
is an admissible Young measure solution to \eqref{e:weak-setting} provided the energy-tension bound \begin{equation}\label{admc} 
 \int_{\Omega} \langle\nu_{t,s},
  | \xi |^{2}\rangle\,ds \,dt +\lambda(\overline\Omega)
  \leq \Theta 
\end{equation} holds, where $\Theta$ is a certain constant depending only on $T$, $g$, and the $L^2$-norms of $\alpha$ and $\beta$, and 
\begin{eqnarray}\label{d:YMS}
&& \int_{\Omega} \langle \nu_{t,s} , \mathcal{A}(\xi) \rangle \varphi_{t}(t,s)\,ds \,dt  = \int_{\Omega} \langle \nu_{t,s} , \mathcal{B}(\xi)\rangle \varphi_{s}(t,s)\,ds \,dt  \\
&& \qquad \qquad \qquad  + \int_{\overline{\Omega}} \langle \nu_{t,s}^{\infty}, \mathcal{B}^{\infty}(\theta) \rangle \varphi_{s}(t,s) \lambda(d t, d s) + \Xi_{0}(\alpha, \beta, \varphi) \nonumber
\end{eqnarray}
for every $\varphi \in \tilde{C}^{\infty}(\overline{\Omega})$. 
\end{defn}

\begin{rem} {\rm If an admissible Young measure solution $(\nu, \lambda, \nu^{\infty})$ satisfies $\,\nu_{t,s}=\delta_{\gamma(t,s)}$ a.e. in $\Omega$, where $\gamma:\Omega\to \R^6$ is a measurable function and $\delta$ is the Dirac delta, and $\lambda=0$, then $\gamma$ belongs to $L^2(\Omega;\mathbb{R}^6)$ and is a weak solution in the sense of \eqref{e:merged}. Assume now that $\gamma$ is a regular function on $\overline\Omega$ and the compatibility conditions \eqref{e-al1}, \eqref{e-al2} and \eqref{e-al3} hold. Then \eqref{e:merged} yields \eqref{e:cont-eqn}, and, as in Section \ref{ss25}, the vector function $\gamma$ generates a pair $(\eta,\sigma)$ satisfying \eqref{e:whip-pde}, \eqref{e:sg0}. Since $\eta$ is determined up to a constant, we can choose it to satisfy $\eta(0,1)=0$. Then, similarly to our previous considerations, we can check that the initial and boundary conditions \eqref{e:ic}, \eqref{e:conds} are met. }
\end{rem}

\begin{rem} {\rm The arguments in Sections \ref{s:wp} and \ref{s:mt} provide a rigorous expression for $\Theta$. } \end{rem}

\begin{rem} {\rm An important open problem is the one of uniqueness of regular solutions to \eqref{e:whip-pde}. The upward whip anomaly (see Sections \ref{s:eqn-sigma} and \ref{s:52}) hints that it should be more rational to study the issue of uniqueness for \eqref{e:whip-pde} coupled with \eqref{e:sg0} (equipped with suitable initial and boundary conditions, either in a strong form, e.g.,  \eqref{e:ic}, \eqref{e:conds}, or in a weak form, e.g.,  \eqref{e:merged}). A positive answer to this question is the cornerstone for such possible developments in the studies of the inextensible string equations as existence of dissipative solutions \cite{blions,v12} and their relation with the Young measure ones, or discovery of additional admissibility constraints in the definition of Young measure solutions which would secure weak-strong uniqueness \cite{Bre-DL-Sz} for \eqref{e:cont-eqn}. }
\end{rem}

%\begin{rem} Our system is intrinsically different from the classical conservation laws and expresses ideal-fluid-like behaviour \cite{tzn}. Therefore, as in the treatments of Euler equations \cite{DiPerna-Majda,blions}, we refrain from imposing additional admissibility constraints in the definition of the solution. 
%\end{rem}

\section{Well-posedness and uniform bounds for the approximate problem}\label{s:wp}

\subsection{Global regularity}

Let $\epsilon \in (0,1]$ be a constant and consider the problem 
\begin{subequations}\label{e:app-pr-v-kappa}
\begin{align}
& v_{t}   = \epsilon v_{ss} + \kappa_{s}+g, \label{e:v_t(eps)-1} \\
& v_{s}   = \left( \epsilon \kappa + \frac{\kappa}{\sqrt{\epsilon + | \kappa |^{2}}}\right)_{t} - \epsilon \left(\epsilon \kappa + \frac{\kappa}{\sqrt{\epsilon + | \kappa |^{2}}} \right)_{ss} \label{e:v_s(eps)},\\
& \kappa \vert_{s = 0}  = 0  \label{e:bc(eps)-s}, \\
& \left(\epsilon \kappa + \frac{\kappa}{\sqrt{\epsilon + | \kappa |^{2}}}\right) \Bigg\vert_{t = 0} = \alpha_{s}, \\
& v \vert_{s = 1} = 0, \\
& v \vert_{t = 0} = \beta, \\
& \left(\epsilon \kappa + \frac{\kappa}{\sqrt{\epsilon + | \kappa |^{2}}}\right)_{s} \Bigg\vert_{s = 1} = 0, \\
& v_{s} \vert_{s = 0} = 0. \label{e:bc(eps)-e} 
\end{align}
\end{subequations}
Denote $\tau = \epsilon \kappa + \frac{\kappa}{\sqrt{\epsilon + | \kappa |^{2}}}.$ Then, $\kappa = G(\tau)$, where $G$ is a function with positive-semidefinite Jacobian matrix, and $G(0) = 0$. Moreover, observe that the eigenvalues of $\nabla G^{-1}
(\kappa)$ are $\epsilon+\frac{\epsilon}{(\epsilon + | 
\kappa |^{2})^{3/2}}$ and $\epsilon+\frac{1}{(\epsilon + | \kappa |
^{2})^{1/2}}$. Thus, the eigenvalues of $\nabla 
G(\tau)$ are \[\Lambda_1(\tau)=\frac 1{\epsilon+(\epsilon + | 
G(\tau) |^{2})^{-1/2}}\geq \frac 1{\epsilon+\epsilon^{-1/2}}\] 
and \[\Lambda_2(\tau)=\frac 
{\epsilon^{-1}}{1+(\epsilon + | G(\tau) |^{2})^{-3/2}}\leq \epsilon^{-1}.\]
In particular, $G$ is globally Lipschitz. 
Observe also that 
\begin{equation*}
|\kappa|\geq 1 \quad \Rightarrow \quad |\tau|\geq \epsilon + (1+\epsilon)^{-1/2}> 1,
\end{equation*} and, consequently, 
\begin{equation}\label{e:intau}
|\tau|\leq 1 \quad \Rightarrow \quad |G(\tau)|< 1.
\end{equation}
We can rewrite the problem \eqref{e:app-pr-v-kappa} as
\begin{subequations}\label{e:appr-pr}
\begin{align}
& v_{t}   = \epsilon v_{ss} + (G(\tau))_{s}+g, \label{e:v_t(eps)-2} \\
& \tau_{t}   = v_{s} + \epsilon \tau_{ss}, \label{e:tau_t} \\
& \tau \vert_{s = 0}  = 0,\quad \tau_{s} \vert_{s = 1} = 0, \\
& v \vert_{s = 1}  = 0,\quad v_{s} \vert_{s = 0} = 0, \\
& \tau \vert_{t = 0}  = \alpha_{s},\quad v \vert_{t = 0} = \beta .
\end{align}
\end{subequations}

\begin{thm}\label{t:exist-app}
Let $\alpha, \beta \in C^{3}([0,1]; \mathbb{R}^{3})$, $\alpha_{s}(0) = 0$, $\alpha_{ss}(1) = 0$, $\beta_{s}(0) = 0$, $\beta(1) = 0$. Then there exists a unique solution $(v, \tau)$ to \eqref{e:appr-pr} in the class $C^{\infty}((0,T] \times [0,1]; \mathbb{R}^{6}) \times C(\overline{\Omega}; \mathbb{R}^{6}).$
\end{thm}
\begin{proof}\emph{(Sketch)}
The term $(G(\tau))_{s}$ can be written as $\nabla G(\tau) \tau_{s} = \tilde{G}(t, x) \tau_{s}$, where $\tilde{G}$ is a bounded matrix-valued function. By \cite[Theorem 2]{Redlinger} we obtain an $L^{\infty}$-bound on the solution. The result follows from \cite[Theorems 14.6 and 15.5, Corollary 14.7]{Amann}.
\end{proof}

\subsection{Uniform energy estimates}
\label{uee}

Hereafter in Section \ref{s:wp} we assume that 
\begin{equation} \label{e:as} 
|\alpha_s(s)|\leq 1 \quad \text{for} \quad 0\leq s\leq 1,
\end{equation} 
that 
\begin{equation*}
 \alpha \vert_{s = 1} = 0,
 \end{equation*} and that there is a constant $C_*$ such that 
\begin{equation} \label{e:ab} 
\int_{0}^{1} |\alpha|^2(s) \,ds+\int_{0}^{1} |\beta|^2(s) \,ds \leq C_*.
\end{equation}
 Multiplying \eqref{e:v_t(eps)-2} by $v$ and integrating with respect to $s$ gives
\begin{eqnarray*}
& & \int_{0}^{1} v_{t} v \,ds  =  \epsilon \int_{0}^{1} v_{ss} v \,ds+ \int_{0}^{1} (G(\tau))_{s} v \,ds +\int_{0}^{1}g v \,ds \\
& & \quad = - \epsilon \int_{0}^{1} v_{s} v_{s} \,ds - \int_{0}^{1} G(\tau) v_{s} \,ds+\int_{0}^{1}g v \,ds \\
&  & \quad = - \epsilon \int_{0}^{1} v_{s} v_{s} \,ds + \epsilon \int_{0}^{1} G(\tau) \tau_{ss} \,ds - \int_{0}^{1} G(\tau) \tau_{t} \,ds+\int_{0}^{1}g v \,ds .
\end{eqnarray*}
Hence, 
\begin{eqnarray} \label{e:ener1} 
 - \epsilon \int_{0}^{1} v_{s} v_{s} \,ds  &=& \int_{0}^{1} v_{t} v \,ds  -\int_{0}^{1}g v \,ds  \nonumber \\ & & \qquad + \epsilon \int_{0}^{1} \nabla G(\tau) \tau_{s} \tau_{s} \,ds + \int_{0}^{1} G(\tau) \tau_{t} \,ds.
\end{eqnarray}
Considering the last term,
\begin{eqnarray*}
& & \int_{0}^{1} G(\tau) \tau_{t} \,ds =  \int_{0}^{1} \kappa \left(\epsilon \kappa + \frac{\kappa}{\sqrt{\epsilon + | \kappa |^{2}}}\right)_{t} \,ds \\
& & \qquad = \epsilon \int_{0}^{1} \kappa \kappa_{t} \,ds + \int_{0}^{1} \kappa \left(\frac{\kappa}{\sqrt{\epsilon + | \kappa |^{2}}}\right)_{t} \,ds \\
& & \qquad = \epsilon \int_{0}^{1} \kappa \kappa_{t} d\,s +\int_{0}^{1} \kappa \frac{\kappa_{t}}{\sqrt{\epsilon + | \kappa |^{2}}} ds - \int_{0}^{1} \frac{| \kappa |^2  \kappa}{(\sqrt{\epsilon + | \kappa |^{2}})^{3}}\kappa_{t} \,ds \\
& & \qquad = \epsilon \int_{0}^{1} \kappa \kappa_{t} ds + \epsilon \int_{0}^{1} \frac{\kappa \kappa_{t}}{(\sqrt{\epsilon + | \kappa |^{2}})^{3}} \,ds \\
& & \qquad = \epsilon \,\frac{d}{dt} \int_{0}^{1} \Big(\frac{| \kappa |^{2}}{2} - \frac{1}{\sqrt{\epsilon + | \kappa |^{2}}}\Big) \,ds .
\end{eqnarray*}
Let 
\begin{equation}\label{e:eta} 
\eta(t,s) = \alpha(s) + \int_{0}^{t} v(r,s)\,dr
\end{equation} 
and define the energy as
\begin{eqnarray*}
E_{\epsilon}(t) & =& \frac{1}{2} \int_{0}^{1} |v|^{2} \,ds -\int_{0}^{1}g \eta \,ds+ \frac{\epsilon}{2}  \int_{0}^{1} |\kappa|^{2} \,ds \\ 
& & \quad \quad +\sqrt{\epsilon} - \epsilon \int_{0}^{1} \frac{1}{\sqrt{\epsilon + | \kappa |^{2}}} \,ds+ \epsilon \int_{0}^{t}\int_{0}^{1} \nabla G(\tau) \tau_{s} \tau_{s} \,ds\,dt.
\end{eqnarray*}
Then \eqref{e:ener1} yields
\[(E_{\epsilon})_{t} = - \epsilon \int_{0}^{1} v_{s} v_{s} \,ds  \leq 0.\]
The initial energy 
\begin{eqnarray*} 
& & E_{\epsilon}(0) = \frac{1}{2} \int_{0}^{1} |\beta|^{2} \,ds -\int_{0}^{1}g \alpha \,ds+ \frac{\epsilon}{2}  \int_{0}^{1} |G(\alpha_s)|^{2} \,ds  \\ 
& & \qquad \qquad \qquad \qquad +\sqrt{\epsilon}- \epsilon \int_{0}^{1} \frac{1}{\sqrt{\epsilon + | G(\alpha_s)|^{2}}} \,ds.
\end{eqnarray*} 
is bounded due to \eqref{e:intau}, \eqref{e:as}, \eqref{e:ab}.
Therefore,
\begin{eqnarray}\label{e:energy-1}
& & \frac{1}{2} \int_{0}^{1} |v|^{2} \,ds + \frac{\epsilon}{2}  \int_{0}^{1} |\kappa|^{2} \,ds + \epsilon \int_{0}^{t}\int_{0}^{1} \nabla G(\tau) \tau_{s} \tau_{s} \,ds\,dt \nonumber  \\ 
&& \qquad \qquad \qquad \qquad \qquad \leq C+\int_{0}^{1}g \eta \,ds\leq C.
\end{eqnarray}
Note that the second inequality follows from the first one and the Gr\"onwall's lemma since 
\[\frac{d}{dt} \int_{0}^{1}g \eta \,ds = \int_{0}^{1}g v \,ds\leq \frac{1}{2}\int_{0}^{1}|v|^2 \,ds +\frac{1}{2}\int_{0}^{1}|g|^2 \,ds\leq \int_{0}^{1}g \eta \,ds + C.\]
Finally, we deduce that
 \begin{eqnarray}\label{e:energy-2}
 \frac 1 {1+\epsilon^{-3/2}} \int_{0}^{T}\int_{0}^{1} |\tau_{s}|^2 \,ds\,dt & \leq & 
\epsilon \int_{0}^{T}\int_{0}^{1} \Lambda_1(\tau) |\tau_{s}|^2 \,ds\,dt \nonumber \\ 
&\leq& \epsilon \int_{0}^{T}\int_{0}^{1} \nabla G(\tau) \tau_{s} \tau_{s} \,ds\,dt \leq C.
\end{eqnarray}

\subsection{Estimate for the tension}

The estimate obtained in this section, together with the one for kinetic energy, is crucial for the rest of the analysis. We let
\begin{equation}\label{e:zeta} 
\zeta(t,s) = \int_{1}^{s} \tau(t,w)\,dw.
\end{equation}
From \eqref{e:tau_t} we find \begin{equation*}  \tau(t,s) = \eta_{s}(t,s) + \epsilon \int_{0}^{t} \tau_{ss}(r,s) \,dr.\end{equation*} Consequently, 
 \begin{equation}  \label{e:sig1}
  \zeta(t,s) = \eta(t,s) + \epsilon \int_{0}^{t} \tau_{s}(r,s) \,dr. 
\end{equation} 
From \eqref{e:eta} we get
\begin{equation}  \label{e:sig2}
 (|\eta|^{2})_{tt} = 2 \eta_{tt} \eta + 2 \eta_{t} \eta_{t} = 2 v_{t} \eta + 2 |v|^{2},
 \end{equation} 
and from \eqref{e:v_t(eps)-2} we obtain
\begin{eqnarray} \label{e:sig3}
\int_{0}^{1} v_{t} \zeta \,ds & = & \epsilon \int_{0}^{1} v_{ss} \zeta \,ds + \int_{0}^{1} (G(\tau))_{s} \zeta \,ds+\int_{0}^{1} g \zeta \,ds  \nonumber \\
& = & - \epsilon \int_{0}^{1} v_{s} \tau \,ds - \int_{0}^{1} G(\tau) \tau \,ds +\int_{0}^{1} g \zeta \,ds .
\end{eqnarray}
Combining  \eqref{e:zeta} -- \eqref{e:sig3}, we infer
\begin{eqnarray}\label{e:estimate} 
& &  \int_{0}^{1} G(\tau) \tau\,ds  \nonumber = \\
&& \quad =  -\epsilon\int_{0}^{1} v_{t} \left[\int_{0}^{t} \tau_{s}(r,s) \,dr\right] \,ds - \int_{0}^{1} \left(\frac{|\eta|^{2}}{2}\right)_{tt} \,ds + \int_{0}^{1} |v|^{2} \,ds  \nonumber \\ 
 & & \qquad \qquad - \epsilon \int_{0}^{1}v_{s} \tau\,ds +\int_{0}^{1} g \eta \,ds +\epsilon\int_{0}^{1} g \left[\int_{0}^{t} \tau_{s}(r,s) \,dr\right] \,ds \nonumber \\
 && \quad = : I_1(t)+I_2(t)+I_3(t)+I_4(t)+I_5(t)+I_6(t).
\end{eqnarray}
The time integral of the first integral is 
\begin{eqnarray}\label{e:i1}
\int_{0}^{T} I_1(t)\,dt  &=& - \epsilon\int_{0}^{T}\int_{0}^{1} v_{t} \left[\int_{0}^{t} \tau_{s}(r,s) \,dr\right] \,ds \, dt \nonumber \\ 
&=& - \epsilon \int_{0}^{1} v\Big\vert_{t=T} \left[\int_{0}^{T} \tau_{s}(r,s) \,dr\right] \,ds  +\epsilon\int_{0}^{T}\int_{0}^{1} v  \tau_{s}\,ds \, dt \nonumber \\
& \leq & \frac{1}{2}\int_{0}^{1} |v|^2\Big\vert_{t=T}\,ds + \frac{\epsilon^2}{2} \int_{0}^{1} \left[\int_{0}^{T} \tau_{s}(r,s) \,dr\right]^2 \,ds \nonumber\\ 
& & \qquad +\frac{1}{2}\int_{0}^{T}\int_{0}^{1} |v|^2\,ds \, dt +\frac{\epsilon^2}{2}\int_{0}^{T}\int_{0}^{1} |\tau_{s}|^2\,ds \, dt.
\end{eqnarray}
The first and third terms are bounded by the energy estimate \eqref{e:energy-1}, and the second and the fourth ones are bounded by $C\epsilon^2(1+\epsilon^{-3/2})$ due to \eqref{e:energy-2}.

For the second integral in \eqref{e:estimate} we have
\begin{equation*}
\begin{split}
\int_{0}^{T} I_2(t)\,dt & = -\int_{0}^{T} \int_{0}^{1} \left(\frac{|\eta|^{2}}{2}\right)_{tt} \,ds\, dt \\ & = -\int_{0}^{1}  \left(\frac{|\eta|^{2}}{2}\right)_{t}\Big\vert_{t = T}\, ds + \int_{0}^{1}  \left(\frac{|\eta|^{2}}{2}\right)_{t}\Big\vert_{t = 0} \,ds \\
& = -\int_{0}^{1} \eta \eta_{t}\Big\vert_{t = T}\,ds + \int_{0}^{1} \alpha \beta \,ds \\
& \leq \frac{1}{2}\int_{0}^{1} |\eta_{t}|^{2}\Big\vert_{t = T} \,ds + \frac{1}{2}\int_{0}^{1} |\eta|^{2}\Big\vert_{t = T} \,ds +  \int_{0}^{1} \alpha \beta \,ds.
\end{split}
\end{equation*}
Here, the first integral is bounded by \eqref{e:energy-1}; the second integral is bounded since the linear operator $v\mapsto \eta$, i.e., $v(t)\mapsto \alpha+\int_{0}^{t} v(r)\,dr$, is bounded in the Banach space $L^{\infty}(0,T; L^{2}(0,1;\mathbb{R}^3))$; the third integral is bounded due to \eqref{e:ab}.

Continuing from \eqref{e:estimate}, $I_3$ and $I_5$ are bounded by the energy bound \eqref{e:energy-1}, and \[\int_{0}^{T} I_4(t)\,dt=\epsilon \int_{0}^{T}\int_{0}^{1}v \tau_s\,ds\,dt\leq C+C\epsilon^2(1+\epsilon^{-3/2})\] as in \eqref{e:i1}. Finally, 
\begin{eqnarray}\label{e:i6}
\int_{0}^{T} I_6(t)\,dt  & =&  \epsilon\int_{0}^{T}\int_{0}^{1} g \left[\int_{0}^{t} \tau_{s}(r,s) \,dr\right] \,ds \, dt \nonumber \\ 
&=&  -\epsilon\int_{0}^{T}\int_{0}^{1} (t-T)g  \tau_{s}\,ds \, dt \nonumber \\ 
& \leq& \frac{1}{2}\int_{0}^{T}\int_{0}^{1} |(t-T)g|^2\,ds \, dt +\frac{\epsilon^2}{2}\int_{0}^{T}\int_{0}^{1} |\tau_{s}|^2\,ds \, dt \nonumber \\ 
& \leq& C+C\epsilon^2(1+\epsilon^{-3/2}).
\end{eqnarray}
Therefore, from \eqref{e:estimate} we conclude that 
\begin{equation*}
\int_{0}^{T} \int_{0}^{1} G(\tau) \tau \,ds\leq C,
\end{equation*}
whence
\begin{equation*}
\int_{0}^{T} \int_{0}^{1} \kappa \left(\epsilon \kappa + \frac{\kappa}{\sqrt{\epsilon + | \kappa |^{2}}}\right) \,ds\, dt \leq C. 
\end{equation*}
Thus, 
\begin{eqnarray}
\int_{\Omega} |\kappa(t,s)| \,ds\, dt &\leq&
C+\int_{\Omega,\ |\kappa|\geq 1} |\kappa| \,ds\, dt \nonumber \\ 
&\leq & C+\int_{\Omega,\ |\kappa|\geq 1} (\epsilon+(1+\epsilon)^{-1/2})|\kappa| \,ds\, dt \nonumber  \\ 
&\leq & C+\int_{\Omega,\ |\kappa|\geq 1}  \left(\epsilon |\kappa| + \frac{|\kappa|}{\sqrt{\epsilon + | \kappa |^{2}}}\right)|\kappa| \,ds\, dt \nonumber \\
& \leq& C+ \int_{\Omega} \kappa \left(\epsilon \kappa + \frac{\kappa}{\sqrt{\epsilon + | \kappa |^{2}}}\right) \,ds\, dt \leq C. 
\end{eqnarray}

\section{Existence of the Young measure solution} \label{s:5}
\subsection{Main theorem} \label{s:mt}
\begin{thm} \label{maint}
Given a pair $\alpha\in Lip_{1}([0,1]; \mathbb{R}^{3}),$ $ \beta \in L^{2}(0,1; \mathbb{R}^{3})$with $\alpha(1) = 0,$ there exists an admissible Young measure solution to \eqref{e:weak-setting}.
\end{thm}
\begin{proof}
Take any sequence $\varepsilon_{n} \to 0.$ The data $(\alpha, \beta)$ can be approximated in $L^{2}(0,1; \mathbb{R}^{6})$ by a sequence of $C^{3}$-functions $(\alpha_{n}, \beta_{n})$ such that $|(\alpha_n)_s(s)|\leq 1$, $(\alpha_{n})_{s}(0) = 0$, $(\alpha_{n})_{ss}(1) = 0$, $\alpha_{n}(1) = 0$, $(\beta_{n})_{s}(0) = 0$, $\beta_{n}(1) = 0$. By Theorem \ref{t:exist-app} there exist smooth solutions $(v_{n}, \tau_{n})$ to \eqref{e:appr-pr} with $\epsilon = \varepsilon_{n}$, $\alpha = \alpha_{n}$, $\beta = \beta_{n}$. Then $(v_{n}, \kappa_{n})$ where $\kappa_{n} = G(\tau_{n})$  is a smooth solution to \eqref{e:app-pr-v-kappa} with $\epsilon = \varepsilon_{n}$, $\alpha = \alpha_{n}$, and $\beta = \beta_{n}$. The uniform energy and tension bounds imply
\begin{align}
& \| v_{n} \|_{L^{\infty}(0, T; L^{2}(0,1))} \leq C, \label{e:v_n-bound} \\
& \| \kappa_{n} \|_{L^{1}(\Omega)} \leq C. \label{e:kappa_n-bound}
\end{align}
Let \begin{equation}\label{e:w_n}
w_{n} = \frac{\kappa_{n}}{\sqrt{\varepsilon_{n} + | \kappa_{n} |^{2}}} + \frac{\kappa_{n}}{\sqrt{| \kappa_{n} |}}.
\end{equation}
Then
\begin{equation} \label{e:w_n-bound}
\| w_{n} \|_{L^{2}(\Omega)} \leq C.
\end{equation}
Consider the function $h_{\varepsilon_{n}} \colon \overline{\mathbb{R}}_{+} \to \overline{\mathbb{R}}_{+}$ defined as
\begin{equation*}
h_{\varepsilon_{n}}(r) = \frac{r}{\sqrt{\varepsilon_{n} + r^{2}}} + \sqrt{r}.
\end{equation*}
We can easily check that this function is strictly increasing. Thus, there exists the inverse function $h_{\varepsilon_{n}}^{-1} \colon \overline{\mathbb{R}}_{+} \to \overline{\mathbb{R}}_{+}$ which is continuous. Observe that $h_{\varepsilon_{n}}^{-1} (0) = 0.$ Let us introduce the functions $H_{\varepsilon_{n}}, H_{\varepsilon_{n}}^{*} \colon \mathbb{R}^{3} \to \mathbb{R}^{3}$ as
\[ H_{\varepsilon_{n}}(\chi) = \frac{\chi}{| \chi |} h_{\varepsilon_{n}}^{-1}(| \chi |), \,\, H_{\varepsilon_{n}}^{*}(\chi) = \frac{\chi}{| \chi |} \sqrt{h_{\varepsilon_{n}}^{-1}(| \chi |)}, \,\,H_{\varepsilon_{n}}(0) = H_{\varepsilon_{n}}^{*}(0) = 0.\]
Observe that these functions are continuous at zero (in fact everywhere). From \eqref{e:w_n} we find that 
\[\kappa_{n} = H_{\varepsilon_{n}}(w_{n})\quad \text{and} \quad \frac{\kappa_{n}}{\sqrt{\varepsilon_{n} + | \kappa |^{2}}} = w_{n} - H_{\varepsilon_{n}}^{*}(w_{n}).\]
Now, \eqref{e:v_t(eps)-1} and \eqref{e:v_s(eps)} imply 
\begin{equation} \label{e:v_n_t}
(v_{n})_{t} = \varepsilon_{n}(v_{n})_{ss} + (H_{\varepsilon_{n}}(w_{n}))_{s} +  g,   \end{equation} and 
\begin{multline}
(v_{n})_{s} = (\varepsilon_{n} H_{\varepsilon_{n}}(w_{n}) + w_{n} - H_{\varepsilon_{n}}^{*}(w_{n}))_{t} \\ - \varepsilon_{n} (\varepsilon_{n} H_{\varepsilon_{n}} (w_{n}) + w_{n} - H_{\varepsilon_{n}}^{*}(w_{n}))_{ss} . \label{e:v_n_s}
\end{multline}
We need the following result to proceed.
\begin{lem}\label{l:H-conv}
We have
\[H_{\varepsilon_{n}}(\chi) \to H_{0}(\chi), \quad H_{\varepsilon_{n}}^{*}(\chi) \to H_{0}^*(\chi)\]
uniformly on $\mathbb{R}^{3}.$
\end{lem}
\begin{proof}
Suppose there exists sequences $\varepsilon_{n_{k}}$ and $\chi_{k}$ such that
\begin{equation*}
| H_{\varepsilon_{n_{k}}}(\chi_{k}) - H_{0}(\chi_{k}) | \geq \delta
\end{equation*}
for some $\delta > 0.$ In the sequel we write ${\varepsilon_{k}}$ instead of ${\varepsilon_{n_{k}}}.$ Due to the above inequality, we get
\begin{equation}\label{e:inv-est}
| h_{\varepsilon_{k}}^{-1}( | \chi_{k} | ) - h_{0}^{-1}(| \chi_{k} |) | \geq \delta.
\end{equation}
Without loss of generality, there exists $\overline\chi  = \underset{k \to \infty}{\lim} | \chi_{k} |$, which can be equal to $+ \infty.$ Assume first that $\overline\chi  \leq 1.$ Then $h_{0}^{-1}(\overline\chi ) = 0$, and, since $ h_{0}^{-1}(| \chi_{k} |)$ is non-negative, we must have $d_{k} :=  h_{\varepsilon_{k}}^{-1}(|\chi_{k}|)  \geq \delta$ for $k$ large enough. Therefore,
\begin{align}
| \chi_{k} | & = h_{\varepsilon_{k}}(d_{k}) = \frac{d_{k}}{\sqrt{\varepsilon_{k} + d_{k}^{2}}} + \sqrt{d_{k}} \nonumber \\
& \geq \frac{\delta}{\sqrt{\varepsilon_{k} + \delta^{2}}} + \sqrt{\delta} \to 1 + \delta \nonumber
\end{align}
which contradicts the assumption $\overline\chi \leq 1.$ Now, consider the case $\overline\chi  >1$. Then without loss of generality $| \chi_{k} | > 1$ for all $k.$ Denote $r_{k} = h_{0}^{-1}(| \chi_{k} |).$ Then, there exist numbers $k_{l}$ for $l = 1, 2, \ldots,$ such that either $r_{k_{l}} \geq d_{k_{l}}$ or $ r_{k_{l}} \leq d_{k_{l}}$ for all $l.$ To simplify the notation, we write $r_{k}$ and $d_{k}$ instead of $r_{k_{l}}$ and $d_{k_{l}}$. Due to \eqref{e:inv-est} we either have $r_{k} \geq d_{k} + \delta$ or $d_{k} \geq r_{k} + \delta.$ In the first case, we have
\begin{align}
\frac{d_{k}}{\sqrt{\varepsilon_{k} + d_{k}^{2}}} + \sqrt{d_{k}} & = 1 + \sqrt{r_{k}} \geq 1 + \sqrt{d_{k} + \delta} \nonumber \\
& \geq \frac{d_{k}}{\sqrt{\varepsilon_{k} + d_{k}^{2}}} + \sqrt{d_{k} + \delta}  \nonumber \\
& > \frac{d_{k}}{\sqrt{\varepsilon_{k} + d_{k}^{2}}} + \sqrt{d_{k}}  \nonumber
\end{align}
which gives a contradiction. In the second case we have 
\begin{align}
1 + \sqrt{r_{k}} & = \frac{d_{k}}{\sqrt{\varepsilon_{k} + d_{k}^{2}}} + \sqrt{d_{k}} \nonumber \\
& \geq \frac{r_{k} + \delta}{\sqrt{\varepsilon_{k} + (r_{k} + \delta)^{2}}} + \sqrt{r_{k} + \delta}. \nonumber 
\end{align}
However, the last inequality cannot hold for all $k$ since by Lagrange's mean value theorem we have
\[\sqrt{r_{k} + \delta} - \sqrt{r_{k}} = \frac{1}{2 \sqrt{r_{k} + c_{0}}} \delta \geq \frac{\delta}{2 \sqrt{r_{k} + \delta}}\]
for some $c_{0} \in (0, \delta),$ and 
\begin{align}
1 - \frac{r_{k} + \delta}{\sqrt{\varepsilon_{k} + (r_{k} + \delta)^{2}}}   & = \frac{\sqrt{\varepsilon_{k} + (r_{k} + \delta)^{2}} - \sqrt{(r_{k} + \delta)^{2}}}{\sqrt{\varepsilon_{k} + (r_{k} + \delta)^{2}}} \nonumber \\
& = \frac{\frac{\varepsilon_{k}}{2 \sqrt{(r_{k} + \delta)^{2} + c_{k}}}}{\sqrt{\varepsilon_{k} + (r_{k} + \delta)^{2}}} \leq \frac{\varepsilon_{k}}{2 (r_{k} + \delta)^{2}}  \nonumber \\
& < \frac{\delta}{2 \sqrt{r_{k} + \delta}} \nonumber
\end{align}
for $k$ large enough and some $c_{k} \in (0, \varepsilon_{k}).$ Similarly, one shows that the inequality
\[| H_{\varepsilon_{n_{k}}}^{*}(\chi_{k}) - H_{0}^{*}(\chi_{k})  | \geq \delta\]
cannot hold. 
\end{proof}

We now return to the proof of the theorem and introduce the functions
$\mathcal{A}_{\varepsilon}, \mathcal{B}_{\varepsilon},$ $\mathcal{D}_{\varepsilon},  \mathcal{D} \colon \mathbb{R}^{6} \to \mathbb{R}^{6}$ as
\begin{align}
& \mathcal{A}_{\varepsilon}(\gamma) = \mathcal{A}_{\varepsilon}(v, w) = (v, w - H_{\varepsilon}^{*}(w)), \nonumber \\
& \mathcal{B}_{\varepsilon}(\gamma) = (H_{\varepsilon}(w), v), \nonumber  \\
& \mathcal{D}_{\varepsilon}(\gamma) = (0, H_{\varepsilon}(w)), \nonumber \\
& \mathcal{D}(\gamma) = (0, H_{0}(w)). \nonumber
\end{align}
Let $\gamma_{n} = (v_{n}, w_{n}).$ Then, \eqref{e:v_n_t} and \eqref{e:v_n_s} may be rewritten as 
\begin{eqnarray}\label{e:AB-eps}
&& (\mathcal{A}_{\varepsilon_{n}})_{t}(\gamma_{n}) + \varepsilon_{n} (\mathcal{D}_{\varepsilon_{n}})_{t}(\gamma_{n}) = \\
&& \qquad  \quad  = (\mathcal{B}_{\varepsilon_{n}})_{s}(\gamma_{n}) + \varepsilon_{n} (\mathcal{A}_{\varepsilon_{n}})_{ss}(\gamma_{n}) + \varepsilon_{n}^{2} (\mathcal{D}_{\varepsilon_{n}})_{ss}(\gamma_{n}) + (g, 0). \nonumber
\end{eqnarray}
Moreover, by \eqref{e:v_n_t} and \eqref{e:v_n_s}, the initial and boundary conditions \eqref{e:bc(eps)-s} - \eqref{e:bc(eps)-e}, and the restriction $\alpha_{n}(1) = 0$, we find that for any $\varphi = (\phi, \psi) \in \tilde{C}^{\infty}(\Omega)$ we have
\begin{eqnarray*}
&& \int_{\Omega} v_{n} \phi_{t} \,ds\,dt  = \int_{\Omega} H_{\varepsilon_{n}}(w_{n}) \phi_{s}\,ds\,dt - \int_{0}^{1} \beta_{n} \phi \vert_{t = 0}\,ds \\
&& \qquad  \qquad \qquad \qquad  - \int_{\Omega} g \phi\,ds\,dt - \varepsilon_{n} \int_{\Omega} v_{n} \phi_{ss}\,ds\,dt, \\
&& \int_{\Omega} [w_{n} - H_{\varepsilon_{n}}^{*}(w_{n}) + \varepsilon_{n} H_{\varepsilon_{n}}(w_{n})] \psi_{t} \,ds\,dt =  \int_{\Omega} v_{n} \psi_{s}\,ds\,dt \\
& & \qquad  + \int_{0}^{1} \alpha_{n} \psi_{s}\vert_{t = 0} \,ds - \varepsilon_{n} \int_{\Omega} [w_{n} - H_{\varepsilon_{n}}^{*}(w_{n}) + \varepsilon_{n} H_{\varepsilon_{n}}(w_{n})] \psi_{ss}\,ds\,dt.
\end{eqnarray*}
These can be merged to give
\begin{align}
& \int_{\Omega} \mathcal{A}_{\varepsilon_{n}}(\gamma_{n}) \varphi_{t}\,ds\,dt + \varepsilon_{n} \int_{\Omega} \mathcal{D}_{\varepsilon_{n}}(\gamma_{n}) \varphi_{t}\,ds\,dt  = \nonumber \\
&\qquad \qquad  \qquad =  \int_{\Omega} \mathcal{B}_{\varepsilon_{n}}(\gamma_{n}) \varphi_{s}\,ds\,dt - \varepsilon_{n} \int_{\Omega} \mathcal{A}_{\varepsilon_{n}}(\gamma_{n}) \varphi_{ss}\,ds\,dt  \nonumber \\
& \qquad \qquad \qquad \qquad - \varepsilon_{n}^{2} \int_{\Omega} \mathcal{D}_{\varepsilon_{n}}(\gamma_{n}) \varphi_{ss}\,ds\,dt + \Xi_{0} (\alpha_{n}, \beta_{n}, \varphi). \label{e:merged-ABD}
\end{align}
Due to \eqref{e:v_n-bound} and \eqref{e:w_n-bound} we have
\begin{equation}\label{e:gamma_n-bound}
\| \gamma_{n} \|_{L^{2}(\Omega; \mathbb{R}^{6})} \leq C.
\end{equation} Observe that this constant merely depends on $T$, $g$, and the $L^2$-norms of $\alpha$ and $\beta$.

By Lemma \ref{l:H-conv} we obtain
\begin{align}
& \mathcal{A}_{\varepsilon_{n}}(\gamma) \to \mathcal{A}(\gamma) \label{e:A-conv}, \\
& \mathcal{B}_{\varepsilon_{n}}(\gamma) \to \mathcal{B}(\gamma) \label{e:B-conv}, \\
& \mathcal{D}_{\varepsilon_{n}}(\gamma) \to \mathcal{D}(\gamma),\label{e:D-conv}
\end{align}
uniformly in $\gamma \in \mathbb{R}^{6}.$ From \eqref{e:merged-ABD} we infer 
\begin{equation}\label{e:appr-eq}
\begin{split}
& \int_{\Omega} \mathcal{A}(\gamma_{n}) \psi_{t}\,ds\,dt - \int_{\Omega} \mathcal{B}(\gamma_{n}) \varphi_{s}\,ds\,dt - \Xi_{0}(\alpha, \beta, \varphi)  = \\
& \qquad = \int_{\Omega} [ \mathcal{A}(\gamma_{n}) - \mathcal{A}_{\varepsilon_{n}}(\gamma_{n})] \varphi_{t}\,ds\,dt  \\
& \qquad \quad  + \int_{\Omega} [ \mathcal{B}_{\varepsilon_{n}}(\gamma_{n}) - \mathcal{B}(\gamma_{n})] \varphi_{s}\,ds\,dt  \\ 
& \qquad  \quad - {\varepsilon_{n}} \int_{\Omega} [\mathcal{D}_{\varepsilon_{n}}(\gamma_{n}) - \mathcal{D}(\gamma_{n})] \varphi_{t}\,ds\,dt  \\
& \qquad \quad - {\varepsilon_{n}} \int_{\Omega} [ \mathcal{A}_{\varepsilon_{n}}(\gamma_{n}) - \mathcal{A}(\gamma_{n})] \varphi_{ss}\,ds\,dt  \\
& \qquad \quad - \varepsilon_{n}^{2} \int_{\Omega} [\mathcal{D}_{\varepsilon_{n}}(\gamma_{n}) - \mathcal{D}(\gamma_{n})] \varphi_{ss}\,ds\,dt  \\
& \qquad \quad - {\varepsilon_{n}} \int_{\Omega} \mathcal{D}(\gamma_{n}) \varphi_{t}\,ds\,dt - {\varepsilon_{n}} \int_{\Omega} \mathcal{A}(\gamma_{n}) \varphi_{ss}\,ds\,dt  \\
& \qquad \quad - \varepsilon_{n}^{2} \int_{\Omega} \mathcal{D}(\gamma_{n}) \varphi_{ss}\,ds\,dt + \Xi_{0}(\alpha_{n} - \alpha, \beta_{n} - \beta, \varphi). 
\end{split}
\end{equation}
The first five terms on the right-hand side tend to zero by \eqref{e:A-conv} -- \eqref{e:D-conv}. Since $\mathcal{A}$ and $\mathcal{D}$ are sublinear and subquadratic, respectively, \eqref{e:gamma_n-bound} gives
\begin{align}
& \| \mathcal{A}(\gamma_{n}) \|_{L^{2}(\Omega ; \mathbb{R}^{6})} \leq C, \nonumber \\
& \| \mathcal{D}(\gamma_{n}) \|_{L^{1}(\Omega ; \mathbb{R}^{6})} \leq C . \nonumber
\end{align}
Recall that $\alpha_{n} \to \alpha$, $\beta_{n} \to \beta$ in $L^{2}(0,1; \mathbb{R}^{3})$. Hence, we conclude that the remaining terms on the right-hand side of \eqref{e:appr-eq}  go to zero. Consider the functions 
\begin{align}
& \tilde{\mathcal{A}}(t, s, \xi) = \mathcal{A}(\xi) \varphi_{t}(t,s), \nonumber \\
& \tilde{\mathcal{B}}(t, s, \xi) = \mathcal{B}(\xi) \varphi_{s}(t,s) .\nonumber
\end{align}
It is easy to see that $\tilde{\mathcal{A}}$ and $\tilde{\mathcal{B}}$ are in the class $\mathcal{F}_{2},$ $\tilde{\mathcal{A}}^{\infty} \equiv 0,$ and $\tilde{\mathcal{B}}^{\infty}(t, s, \xi) = \mathcal{B}^{\infty}(\xi) \varphi_{s}(t, s).$ By Theorem \ref{t:Young}, we can pass to the limit in \eqref{e:appr-eq} (passing to a subsequence, if necessary) and obtain
\begin{eqnarray}
& &  \int_{\Omega} \langle \nu_{t,s}, \tilde{\mathcal{A}}(t, s, \xi) \rangle  \,ds\,dt - \int_{\Omega} \langle \nu_{t,s} , \tilde{\mathcal{B}}(t, s, \xi) \rangle \,ds\,dt \nonumber \\
& & \quad \quad   - \int_{\overline{\Omega}} \langle \nu_{t,s}^{\infty}, \tilde{\mathcal{B}}^{\infty}(t, s, \theta)\rangle\,\lambda(dt, ds) 
  - \Xi_{0}(\alpha, \beta, \varphi)=0,
\end{eqnarray}
which yields \eqref{d:YMS}. Remark \ref{remme} and \eqref{e:gamma_n-bound} imply \eqref{admc}.
\end{proof}

\subsection{Examples} \label{s:52} Let us briefly examine the implications of Theorem \ref{maint} for some particular cases of chain dynamics with the ``whip'' boundary conditions and non-zero gravity $g$. In the case of the initial data \eqref{e:unst}, we get existence of a generalized solution which is a priori different from the stationary solution \eqref{e:ss} plainly because the latter one does not admit non-negative tension. A qualitative glimpse at the auxiliary problems \eqref{e:app-pr-v-kappa} and \eqref{e:appr-pr} implies that the ``approximate strings'' start to evolve close to the upright position \eqref{e:unst} but eventually with the course of time they approach their steady-states. As $\epsilon$ goes to zero, these steady-states approach the downwards vertical orientation with \begin{equation}\label{e:ss1} v(s)=0,\ \kappa(s)=-gs.\end{equation}  Hence, our solution must be relevant in connection with the problem of falling of a chain which is initially in an upright position and then its upper end is released and the lower one remains fixed. 

On the other hand, there are many physical and mechanical works dealing with a problem of falling of a chain which initially has two ends together and then one of them is released (see \cite{Wong} for a review). In this case, the initial data are \begin{equation} \label{e:unstwo} \alpha(s)=g\left(\frac 1 {2|g|} - \left|\frac s {|g|}-\frac 1 {2|g|}\right|\right),\,\beta(s)=0. \end{equation}
Although the compatibility condition \eqref{e-al1} is violated for $s=\frac 1 2$, the hypothesis of Theorem \ref{maint} is met. Thus, the Young measure solution exists, providing a new framework for a correct description of this mechanical system. 

\bibliographystyle{abbrv}
\bibliography{bibliography-sv}

\end{document}